\documentclass[12pt]{amsart}
\usepackage{amsthm,amsfonts, amssymb, amscd}
\newcommand{\Q}{{\mathbb Q}}
\newcommand{\Pro}{{\mathbb P}}
\newcommand{\Z}{{\mathbb Z}}

\newcommand{\F}{{\mathcal F}}

\newcommand{\E}{{\mathcal E}}
\newcommand{\V}{{\mathcal V}}

\DeclareMathOperator{\PGL}{PGL}

\DeclareMathOperator{\GL}{GL}

\DeclareMathOperator{\For}{For}

\newtheorem{thm}{Theorem}

\newtheorem{prop}[thm]{Proposition}
\newtheorem{lemma}[thm]{Lemma}
\newtheorem{lemma-definition}[thm]{Lemma-Definition}

\newtheorem{cor}[thm]{Corollary}
\theoremstyle{remark}
\newtheorem{remark}[thm]{Remark}
\newtheorem{example}[thm]{Example}

\begin{document}
\numberwithin{thm}{section}
\title[The forgetful map in rational $K$-theory]
{The forgetful map in rational $K$-theory}

\author{William Graham}
\address{
Department of Mathematics,
University of Georgia,
Boyd Graduate Studies Research Center,
Athens, GA 30602
}
\email{wag@math.uga.edu}
\thanks{The author was supported by the National Science Foundation}

\begin{abstract}
Let $G$ be a connected reductive algebraic group acting on a scheme $X$.  
Let $R(G)$ denote the representation ring of $G$, and
$I \subset R(G)$ the ideal of virtual representations of
rank $0$.  Let $G(X)$
(resp.~$G(G,X)$)
denote the Grothendieck group of 
coherent sheaves (resp.~$G$-equivariant coherent sheaves)
on $X$.  Merkurjev proved that if $\pi_1(G)$ is
torsion-free, then
the forgetful map $G(G,X) \to G(X)$ induces an
isomorphism $G(G,X) /I G(G,X)  \to G(X) $.
Although this map need not be an isomorphism if $\pi_1(G)$ has torsion,
we prove that without the assumption on $\pi_1(G)$, the
map $G(G,X)/I G(G,X) \otimes \Q \to G(X) \otimes \Q$ is an isomorphism.
\end{abstract}

\maketitle

\section{Introduction}
Let $G$ be a connected reductive algebraic group acting on a scheme $X$.
The $G$-equivariant coherent sheaves on $X$ are
central to the study of $X$.   These sheaves often have
computable invariants, since the group action allows
the use of tools
such as localization theorems.  Also, equivariant sheaves are an important source
of sheaves on quotients by group actions, since
if a 
quotient $X \to Y$ exists, then the sheaf of invariant sections of
an equivariant sheaf on $X$ is a coherent sheaf on $Y$. 
It is natural to ask which coherent sheaves on $X$ admit
$G$-actions.
One positive result is due to
Mumford, who proved that 
if $G$ is connected and $X$ is normal, and $L$ is any
invertible sheaf on $X$, then some power of $L$ is $G$-linearizable
\cite[Corollary 1.6]{MFK:94}.  On the other hand, 
it is easy to find examples of coherent
sheaves which do not admit $G$-actions.   For example, 
 $\PGL(2)$ acts on $\Pro^1$ but 
the sheaf ${\mathcal O}_{\Pro^1}(1)$
does not admit an action
of $\PGL(2)$ (see \cite[p.~33]{MFK:94}).   

Merkurjev proved that from the point of view of $K$-theory,
there is no obstruction to equivariance, as long as the
fundamental group of $G$ is torsion-free (see \cite{Mer:97}).
Let $G(X)$
(resp.~$G(G,X)$)
denote the Grothendieck group of 
coherent sheaves (resp.~$G$-equivariant coherent sheaves)
on $X$.  There is a forgetful map $G(G,X) \to G(X)$.
Let $R = R(G)$ denote the representation ring of $G$, and $I \subset R$
the augmentation ideal, that is, the ideal of virtual representations
of rank $0$.  The Grothendieck group $G(G,X)$ is an $R$-module.
Merkurjev showed that if $\pi_1(G)$ is torsion-free, then
the forgetful map induces an isomorphism
$$
G(G,X) / I G(G,X) \to G(X).
$$
If $\pi_1(G)$ is not torsion-free, this map can fail
to be an isomorphism.
For example, the fundamental group of $\PGL(2)$ is
$\Z/2\Z$, and the
 class $v = [{\mathcal O}_{\Pro^1}(1)] \in G(\Pro^1)$
is not in the image of $G(\PGL(2),\Pro^1)$.  
However, if
we tensor with $\Q$, this class is in the image.  Indeed, 
$G(\Pro^1) = \Z[v]/ \langle v^2 - 1 \rangle$, so after tensoring
with $\Q$, we have $v = \frac{1}{2}(v^2+1)$.  This element is in the
image of the forgetful map since $v^2$ is the class of ${\mathcal O}_{\Pro^1}(2)$, which 
has a $G$-action.  

This phenomenon holds more generally:

\begin{thm} \label{t.mainthm}
Let $G$ be a connected reductive algebraic group acting on a scheme
$X$.  The forgetful map $G(G,X) \to G(X)$ induces an isomorphism
 $G(G,X) /I G(G,X) \otimes \Q \to G(X) \otimes \Q$.  Hence the map
 $G(G,X) \otimes \Q \to G(X) \otimes \Q$ is surjective.
 \end{thm}

Merkurjev proves his theorem by using a spectral sequence
relating equivariant and ordinary $K$-theory.
The approach taken in this paper is different,
and makes use of Brion's analogue of Theorem
\ref{t.mainthm} for Chow groups, along 
 with the equivariant Riemann-Roch theorem
 proved by Edidin and the author.   This use of Riemann-Roch
 explains the rational coefficients in the statement of our theorem. 
 
 We remark that Theorem \ref{t.mainthm} remains true even if $G$ is not reductive, provided that
 $G$ has a Levi factor $L$ (which is automatic in characteristic
 $0$), since then the forgetful maps from $G$-equivariant $K$-theory
 and Chow groups to the corresponding $L$-equivariant groups
 are isomorphisms.   Also, we expect that 
a topological version of Theorem \ref{t.mainthm} holds
 for equivariantly formal spaces (since for these spaces
 the map from equivariant cohomology to ordinary cohomology
 is surjective).  Finally, the completion theorem of
 \cite{EdGr:07} should have implications
 in this setting.

 \medskip
 
 {\em Conventions:}  We work over an algebraically closed field $k$,
 and assume that the $G$-actions are locally linear---that is, the schemes on which $G$ acts can be 
 covered by $G$-invariant quasi-projective
 open subsets.  This assumption is automatically satisfied for normal schemes.
 (We work in this setting in order to apply Brion's results, which are
 proved under these hypotheses.  We remark that
  that Merkurjev's results, suitably stated, remain valid
 when the ground field is not algebraically closed.)  
 Also, to make use of functorial properties of
 Riemann-Roch (see \cite[Theorem 18.3(4)]{Ful:84})
 we will assume that our schemes can be equivariantly embedded
 in smooth schemes.

\section{Equivariant K-theory, Chow groups, and Riemann-Roch} \label{s.prelim}
In this section we recall some basic facts about K-theory, Chow groups, and
Riemann-Roch, in the equivariant and non-equivariant settings.
We prove a result comparing topologies on equivariant Chow groups,
and also prove a compatibility result between Riemann-Roch and
forgetful maps.  Both of these results are 
used in the proof of the main theorem.   Our main references
for equivariant Chow groups and equivariant Riemann-Roch
will be \cite{EdGr:98} and \cite{EdGr:00}, where more details
can be found.  If $M$ is an abelian group, we write
$M_{\Q} = M \otimes_{\Z} \Q$.  Because we want to index
Chow groups by codimension, we will assume all schemes and
algebraic spaces considered are equidimensional; our results
are valid without this assumption, but we would have to index
Chow groups by dimension.

We begin with some definitions.
Let $G$ be a linear algebraic group acting on an algebraic space
$X$.  Let $G(G,X)$ (resp.~$G(X)$) 
denote the Grothendieck group of $G$-equivariant
coherent sheaves (resp.~coherent sheaves).   There
is a forgetful map
$$
\For: G(G,X) \to G(X)
$$
which takes the class of a $G$-equivariant coherent sheaf
to the class of the same sheaf, viewed nonequivariantly.  
If we need to keep track of the space involved, we will
denote this by $\For_X$.  Note that
$G(G,X)$ is a module for the representation ring $R = R(G)$
of $G$.  Let $I \subset R$ denote the augmentation ideal
(the ideal of virtual representations of rank $0$).  
Let $\widehat{G(G,X)_{\Q}}$ denote the $I$-adic completion of
$G(G,X)_{\Q}$ (not the tensor product with $\Q$ of the 
$I$-adic completion of $G(G,X)$).

Let $CH^i(X)$ denote the codimension $i$ Chow group of $X$;
if $X$ has pure dimension $d$, then $CH^i(X) = A_{d-i}(X)$.
Write $CH^*(X) = \oplus_i CH^i(X)$.  Similarly, let
$CH^i_G(X) = A_{d-i}^G(X)$ denote the ``codimension $i$"
equivariant
Chow group of $X$, and 
$CH^*_G(X) = \oplus CH^i_G(X)$.
By definition, if $V$ is a representation of $G$ and $U$
an open subset of $V$ on which $G$ acts freely, then
$CH^i_G(X) = CH^i((X \times U)/G)$.  
This definition is independent of the
choice of $V$ and $U$ (see \cite{EdGr:98}).
We will denote the mixed space $(X \times U)/G$ by
$X \times^G U$ or $X_G$.  Now, $X$ is
embedded in $X_G$ as a fiber of the map
$X_G \to U/G)$, and pullback along this
embedding gives a map 
$$
\For: CH^i_G(X) \to CH^i(X).
$$  
Note that $CH^*_G(X)$ is a module for the graded ring
$S = CH^*_G(\mbox{pt})$.
Let $J \subset S$ be the ideal spanned by the 
homogeneous elements of $S$ of positive degree.  

The following proposition is similar to \cite[Prop.~2.1]{EdGr:00},
which dealt with the case where $G$ is a subgroup of the group
of upper triangular matrices.  The proof is a minor modification
of that proof.

\begin{prop} \label{p.topologies}
Let $G$ be a connected reductive algebraic group acting on 
an scheme $X$.  Let $N = CH^*_G(X)_{\Q}$.
The topologies on $N$ induced by the two filtrations
$\{ J^n N \}$ and $\{ \oplus_{i \ge n} N^i \}$ coincide.
\end{prop}

\begin{proof}
We must show two things.  First, given any $n$, there exists
an $r$ such that $J^r N \subset \oplus_{i \ge n} N^i$.  For this
we may take $n=r$, since $N$ is non-negatively graded and
$J N^i \subseteq N^{i+1}$.  Second, given any $n$, there exists
an $r$ such that $\oplus_{i \ge r} N^i \subseteq J^n N$.
Indeed, Brion proved that $N / JN \simeq CH^*(X)$.  Thus,
$N/JN$ is $0$ in degrees greater than $d = \mbox{dim }X$,
so $N^p = J N^{p-1}$ for $p>d$.  Thus, for $p \ge n+d$,
we have $N^p = J^n N^{p-n}$, so for $r = n+d$,
we have $\oplus_{i \ge r} N^i \subseteq J^n N$, as desired.
\end{proof}

\begin{cor} \label{c.topologies}
Let $G$ be a connected reductive algebraic group acting on 
an scheme $X$.  The $J$-adic completion of $CH^*_G(X)_{\Q}$
is isomorphic to the direct product $\prod_{i = 0}^{\infty} CH^i_G(X)_{\Q}$.
\end{cor}

\begin{proof}
Since the completion of $CH^*_G(X)_{\Q}$ with respect to the
topology induced by the
second filtration above is the direct product $\prod_{i = 0}^{\infty} CH^i_G(X)_{\Q}$,
this follows from the preceding proposition.
\end{proof}

In \cite{EdGr:00}, the authors constructed an equivariant
Riemann-Roch
map
$$
\tau_X^G: G(G,X) \to \prod_i CH^i_G(X)_{\Q},
$$
with the same functorial properties as the non-equivariant
Riemann-Roch map $\tau_X: G(X) \to CH^*(X)_{\Q}$
of \cite{Ful:84}.
The equivariant Riemann-Roch map induces an isomorphism
$$
\hat{\tau}_X^G: \widehat{G(G,X)_{\Q}} \to \prod_i CH^i_G(X)_{\Q}.
$$
(In \cite{EdGr:00}, $\hat{\tau}_X^G$ was denoted simply by
$\tau_X^G$.)  Also, there is an equivariant Chern character
map $\mbox{ch}_G: R \to S$ which takes $I$ to $J$
and induces an isomorphism of the $I$-adic
completion $\hat{R}$ of $R$ with the $J$-adic completion $\hat{S}$ of $S$.  
Using $\mbox{ch}_G$
to identify $\hat{R}$ with $\hat{S}$, the functorial properties
of $\hat{\tau}^G_X$ (see \cite[Theorem 3.1(c)]{EdGr:00}) imply
that is an isomorphism
of $\hat{R} = \hat{S}$-modules.

The forgetful maps in $K$-theory and Chow groups are compatible
with the Riemann-Roch maps, by the following proposition.   
Let $\tau_X^{G,i}: G(G,X) \to CH^i_G(X)_{\Q}$
(resp.~$\tau_X^i: G(X) \to CH^i(X)_{\Q}$) be the composition of the map $\tau_X^G$ (resp.~$\tau_X$) with the projection to the component of degree $i$.

\begin{prop} \label{p.forget}
Let $G$ be a linear algebraic group acting on an algebraic space
$X$.  The following diagram commutes:
$$
\begin{CD}
G(G,X)  @> {\tau_X^{G,i}}>>  \prod_i CH^i_G(X)_{\Q} \\
@V{\For}VV                               @VV{\For}V \\
G(X)  @>{\tau_X^i}>> CH^*(X)_{\Q}.
\end{CD}
$$
\end{prop}

\begin{proof}
This can be proved using a change of groups argument
along the lines of \cite[Lemma 4.3]{EdGr:00}.  Here
we give a more direct proof.  
Let $V$ be a representation of $G$ and $U$
an open subset of $V$ on which $G$ acts freely, such that
the codimension of $V - U$ is greater than $i$.  By definition,
$CH^i_G(X) = CH^i(X_G)$, where $X_G = X \times^G U$.  Let $\pi: X \times U \to X$ denote
the projection, and let $q: X \times U \to X_G$ denote the quotient
map.  If $\F$ is a coherent sheaf on
$X_G$, then the pullback sheaf $q^* \F$ on $X \times U$ has a natural
$G$-action.  The assignment $\F \to q^* \F$ gives an
equivalence of categories between the category of coherent 
sheaves on $X \times^G U$ and 
the category of $G$-equivariant coherent sheaves on
$X \times U$ (this follows from Thomason's work \cite{Tho:87}; see
\cite{EdGr:00} for a discussion).  
This equivalence yields an isomorphism
$G(X_G) \to G(G, X \times U)$, which we denote by $q^*$.

Let $u \in U$ and let $v \in U/G$ be the image of $u$.  
Let $j: X \to X \times U$ take $X$ to $X \times \{ u \}$,
and let $k = q \circ j: X \to X_G$.  Then $k$ is the
inclusion of $X$ as the fiber of $X_G \to U/G$ over
$v$.  The normal bundle $N_k$ to $k$ is pulled back from the normal 
bundle to the inclusion of $v$ in $U/G$, so $N_k$ is trivial,
and hence by \cite[Theorem 18.3]{Ful:84},
\begin{equation} \label{e.commute1}
\tau_X \circ k^! = k^* \circ \tau_{X_G}
\end{equation}
as maps $G(X_G) \to CH^*(X)_{\Q}$.

Let $\V$ denote the vector bundle $X \times^G (U \times V) \to 
X_G$.  Define
$$
\rho_U: G(X_G) \to CH^*(X_G)
$$
by
\begin{equation} \label{e.rhoU}
\rho_U (\beta) = \frac{\tau_{X_G}(\beta)}{Td(\V)}.
\end{equation}
Let $\rho^i_U$ be the composition of $\rho_U$ with the projection
onto the $i$-th component.  Then by the definition of the
equivariant Riemann-Roch map (see \cite{EdGr:00}),
$\tau_X^{G,i}$ is the top row of the following diagram:
$$
\begin{CD}
G(G,X) @>{\pi^!}>> G(G, X \times U) 
@>{ (q^*)^{-1}}>>   G(X_G) 
@>{\rho_U^i}>>CH^i(X_G) = CH^i_G(X) \\
 @.  @. @V{k^!}VV    @VV{ k^* = \For }V\\
 @.   @.   G(X)  @>{\tau_X^i}>> CH^*(X).
  \end{CD}
$$
Here $\pi^!$ is the flat pullback in equivariant $K$-theory;
if $\E$ is an equivariant coherent sheaf on $X$ then
$\pi^![\E] = [\pi^* \E]$, where $\pi^* \E$ is the pullback of the sheaf
$\E$.  Also, $k^!$ and $k^*$ are the Gysin morphisms associated
to the regular embedding $k$ (see \cite{Ful:84}).  The pullback along $k$
of the vector bundle $\V$ is trivial, so $k^*(Td(\V)) = 1$.  Hence
\eqref{e.commute1} and \eqref{e.rhoU} imply that the diagram commutes.  
To complete the proof of the proposition, it suffices to show that
\begin{equation} \label{e.commute2}
k^! \circ (q^!)^{-1} \circ \pi^! = \For_X
\end{equation}
as maps $G(G,X) \to G(X)$.  Now, $k^! = j^! q^!$, so the left hand side
of \eqref{e.commute2} is
\begin{equation} \label{e.commute3}
j^! q^! (q^*)^{-1} \pi^!.
\end{equation}
By definition, $(q^*)^{-1}$ takes the class of an equivariant coherent sheaf
$\F$ to the class of a nonequivariant sheaf $\E$ with $q^* \E = \F$.
On the other hand, $q^![\E]$ is the class of $q^* \E$ (viewed
as a nonequivariant coherent sheaf) in $G(X \times U)$.  Thus, the composition
$q^! (q^*)^{-1}$ is $\For_{X \times U}: G(G, X \times U) \to G(X \times U)$.  
Since the forgetful map commutes with flat pullback, \eqref{e.commute3} equals
$$
j^! \circ \For_{X \times U} \circ \pi^! =  j^! \pi^! \circ \For_X
=  (\pi \circ j)^! \circ \For_X= \For_X,
$$
as desired.
\end{proof}

\section{Completions}
The purpose of this section is to prove a simple result
(Lemma \ref{l.completions}) about completions.
This lemma is certainly known (cf.~\cite[p.~247]{Bou:72} 
for finitely generated modules), but because of a lack of a reference
for non-finitely generated modules, a proof is included.

Let $R$ be a Noetherian ring and $I$ an ideal of $R$.  Let $\hat{M}$ denote
the $I$-adic completion of the $R$-module $M$.  
We view $\hat{M}$ as the set of coherent sequences $(m_1, m_2, \ldots )$;
here $m_k \in M/I^k M$, and coherent means that 
for all $k$, the natural map $M/I^{k+1}M \to M/I^k M$ takes 
$m_{k+1}$ to $m_k$.
Since $\hat{I} = I \hat{R}$
\cite[Prop.~10.15]{AM:69}, we
have $\hat{I} \hat{M} = I \hat{R} \hat{M} = I \hat{M}$. 
The composition $M \to \hat{M} \to M/ I \hat{M}$ induces a map
$f: M/IM \to \hat{M}/I \hat{M}$.

\begin{lemma}   \label{l.completions}
Let $R$ be a Noetherian ring and $I$ an ideal of $R$.
For any $R$-module $M$, the map $f: M/IM \to \hat{M} / I \hat{M}$
is an isomorphism.
\end{lemma}

\begin{proof}
The exact sequence $0 \to IM \to M \stackrel{\pi}{\rightarrow} M/IM \to 0$
yields an exact sequence of completions
$$
0 \to \widehat{IM} \to \hat{M} \stackrel{\hat{\pi}}{\rightarrow} \widehat{M/IM} = M/IM \to 0
$$
(see \cite[Cor.~10.3, 10.4]{AM:69}).  The map $\hat{\pi}: \hat{M} \to M/IM$ takes
the coherent sequence
$\mu = (m_1, m_2, \ldots)$ to $m_1$.
We claim that the subspaces $\widehat{IM}$ and $I \hat{M}$
of $\hat{M}$ are equal.  This suffices, for then the map
$p: \hat{M}/I \hat{M} \to M/IM$ (induced from $\hat{\pi}$) is an isomorphism.
Indeed, the map
$f: M/IM \to \hat{M}/I \hat{M}$ is induced from the map $M \to \hat{M}/I \hat{M}$
taking $m$ to $(m_1, m_2, \ldots)$, where we set $m_k = m \mbox{ mod } I^k M$.
Since $p \circ f$ is the identity map of $M/IM$, the claim implies that $f$ is
an isomorphism.

It remains to prove the claim.  As noted above, $\mbox{ker }\hat{\pi}
= \widehat{IM}$.  Clearly $I \hat{M} \subseteq \mbox{ker }\hat{\pi}$, so
we must show the reverse inclusion.

Given an element $\mu = (m_1, m_2, \ldots) \in \hat{M}$, let
$p_k(\mu) = m_k \in M/I^k M$.  Let $a_1, \ldots, a_n$ generate $I$.
Suppose that $\mu \in \mbox{ker }\hat{\pi}$.  We want to show that
$\mu \in I \hat{M}$.  Now, $p_1(\mu) = 0$, and
$p_2(\mu) \in IM/I^2 M$.  Let $\mu^1, \ldots, \mu^n$ be
elements of $M$ such that $\sum a_i \mu^i \mbox{ mod } I^2 M
= p_2(\mu)$.  Let $\hat{\mu}^i$ be the image of
$\mu^i$ under $M \to \hat{M}$, and let
$$
\mu(2)  = \mu - \sum a_i \hat{\mu}^i.
$$
Then $p_1(\mu(2)) = p_2(\mu(2)) = 0$, so $p_3(\mu(2)) \in I^2 M / I^3 M$.  
Let $\mu^{ij}$ be elements of $M$ such that 
$\sum a_i a_j \mu^{ij} \mbox{ mod } I^3 M = p_3(\mu(2))$.  Let
$\hat{\mu}^{ij}$ be the image of $\mu^{ij}$ under $M \to \hat{M}$, and
let
$$
\mu(3)  = \mu(2) - \sum a_i a_j \hat{\mu}^{ij}.
$$
Then $p_i(\mu(3)) = 0$ for $i \le 3$.
Proceeding inductively, suppose
we have $\mu(k) \in \hat{M}$ with $p_i(\mu(k))=0$ for
$i \le k$.  Then 
we can find elements $\mu^J \in M$, where
$J$ runs over the collection of
all $k$-element multisets with elements in $\{ 1,2,\ldots,n \}$, such
that if we define
$$
\mu(k+1) = \mu(k) - \sum_{|J| = k} a^J \hat{\mu}^J,
$$
then we have $p_j(\mu(k+1))=0$ for $j \le k+1$.  (Here 
$|J|$ is the number of elements in $J$, counted with multiplicity;
$a^J = \prod_{j \in J} a_j$,
where each $a_j$ occurs with its multiplicity in $J$; and $\hat{\mu}^J$ is the
image of $\mu^J$ under $M \to \hat{M}$.)  Then
$$
\mu =  \sum_k \sum_{|J| = k}  a^J \mu^J;
$$
that is, the right hand side converges to the element $\mu \in \hat{M}$.  
Let $S_i$ be the collection of multisets whose smallest element is
$i$.  We can rewrite the preceding equation as
$$
\mu = a_1 \sum_{J \in S_1} a^{J - \{ 1 \}} \mu^J + 
a_2 \sum_{J \in S_2} a^{J - \{2 \} } \mu^J + \cdots
+ a_n \sum_{J \in S_n} a^{J - \{n \} } \mu^J.
$$
Each of the
series $\sum_{J \in S_i} a^{J - \{ i \} } \mu^J$ converges to an element of $\hat{M}$, 
so we conclude that $\mu \in I \hat{M}$, as desired.
\end{proof}

\begin{remark}
In the proof of the lemma, 
the claim that $\widehat{IM} = I \hat{M}$ admits a simpler proof
if $M$ is finitely generated.  Indeed, in this case the horizontal maps
in the following commutative diagram are isomorphisms
(\cite[Prop.~10.13]{AM:69}):
$$
\begin{array}{ccc}
\hat{R} \otimes_R IM & \rightarrow & \widehat{IM} \\
\downarrow &   & \downarrow \\
\hat{R} \otimes_R M & \rightarrow & \hat{M} .
\end{array}
$$
The image in $M$ of $\hat{R} \otimes_R IM$ under the upper
(resp.~lower) composition is $\widehat{IM}$ (resp.~$I \hat{M}$),
so $\widehat{IM} = I \hat{M}$ as desired.
\end{remark}

\section{Proof of Theorem \ref{t.mainthm}}
In this section we work with rational coefficients and tensor all Grothendieck
groups and Chow groups with $\Q$.   For simplicity we will omit this
from the notation and simply write, for example, $G(G,X)$ for
$G(G,X)_{\Q}$, or $R$ for $R_{\Q}$.  If $M$ is an $R$-module we will write
$M/I$ for $M/IM$, and if $N$ is an $S$-module we will write
$N/J$ for $N/JN$.  Recall that by Corollary \ref{c.topologies}
we can identify the $J$-adic completion of $CH^*_G(X)$
with the direct product $\prod_{i = 0}^{\infty} CH^i_G(X)$.

By Proposition \ref{p.forget}, we have a commutative diagram
$$
\begin{CD}
G(G,X)  @> {\tau_X^G}>>  \prod_i CH^i_G(X) \\
@V{\For}VV                               @VV{\For}V \\
G(X)  @>{\tau_X}>>  \prod_i CH^i(X).
\end{CD}
$$
Now, $\tau^G_X$ takes $I G(G,X)$ to $J \prod CH^i_G(X)$.  Also, the
forgetful maps factor as
$$
G(G,X) \to G(G,X)/I \to G(X)
$$
and
$$
\prod_i CH^i_G(X) \to ( \prod_i CH^i_G(X) )/J \to CH^*(X).
$$
Therefore, we obtain a commutative diagram
\begin{equation} \label{e.cd1}
\begin{CD}
G(G,X)/I  @> {\bar{\tau}_X^G}>>  (\prod_i CH^i_G(X))/J \\
@VVV                               @VVV \\
G(X)  @>{\tau_X}>>  \prod_i CH^i(X),
\end{CD}
\end{equation}
where $\bar{\tau}_X^G$ is induced from $\tau_X^G$.  
The map $\tau_X$ is an isomorphism (see \cite[Corollary 18.3.2]{Ful:84}).  
We claim that $\bar{\tau}_X^G$ is as well.  Indeed, $\tau_X^G$ factors as
$$
G(G,X) \to \widehat{G(G,X)} \stackrel{\hat{\tau}^G_X}{\rightarrow}
\prod_i CH^i_G(X).
$$
and the map $\hat{\tau}^G_X$ is an isomorphism.  As observed
in Section \ref{s.prelim}, if we use $\mbox{ch}_G$
to identify $\hat{R}$ with $\hat{S}$, then $\hat{\tau}^G_X$ is an isomorphism
of $\hat{R} = \hat{S}$-modules.  
Hence $\hat{\tau}^G_X$ induces an isomorphism
$$
 \widehat{G(G,X)}/I \to \left(\prod_i CH^i_G(X) \right)/J.
 $$
(Here we are using the fact that $\hat{I} = I \hat{R}$, so
$ \widehat{G(G,X)}/ \hat{I} =  \widehat{G(G,X)}/I$; 
similarly, $(\prod_i CH^i_G(X))/\hat{J} = (\prod_i CH^i_G(X))/J)$.
We can write $\bar{\tau}_X^G$ as the composition
$$
G(G,X) / I \to \widehat{G(G,X)}/I \to (\prod_i CH^i_G(X))/J.
$$
Since the second map is an isomorphism, and by Lemma \ref{l.completions}
the first map is an isomorphism as well,
we conclude that $\bar{\tau}_X^G$ is an isomorphism, proving the claim.

Now, we have a commutative diagram
$$
\begin{CD} \label{e.cd1}
CH^*_G(X) @>>> \prod_i CH_G^i(X) \\
@V{\For}VV           @VV{\For}V \\
CH^*(X) @= \prod_i CH^i(X)
\end{CD}
$$
(the bottom equality is because $CH^i(X)$ is zero for $i<0$ or $i> \mbox{dim }X$).  
From this we obtain a commutative diagram
\begin{equation} \label{e.cd2}
\begin{CD}
CH^*_G(X)/J @>>> (\prod_i CH_G^i(X))/J \\
@VVV           @VVV \\
CH^*(X) @= \prod_i CH^i(X).
\end{CD}
\end{equation}
The top map is an isomorphism by 
Corollary \ref{c.topologies} and Lemma \ref{l.completions}, and
Brion proved that the left vertical map is an isomorphism.  Hence,
combining diagrams \eqref{e.cd1} and \eqref{e.cd2}, we obtain
a commutative diagram
\begin{equation} \label{e.cd3}
\begin{CD}
G(G,X)/I  @>>>  CH^*_G(X)/J \\
@VVV                               @VVV \\
G(X)  @>{\tau_X}>>   CH^*(X),
\end{CD}
\end{equation}
Since the top, bottom, and right vertical maps are isomorphisms,
we conclude that the left vertical map is an isomorphism as well.
This completes the proof.

\begin{example}
We return to the example of $G = \PGL(2)$ acting on $\Pro^1$,
considered in the introduction.  Let $B$ denote the stabilizer in
$G$ of the point $[1:0]$ and let $T$ denote the maximal torus which is the
image of the diagonal matrices in $\GL(2)$ under the quotient
map $\GL(2) \to \PGL(2)$.  Then $\Pro^1$ can be identified
with $G/B$, and a standard change of groups
argument (see e.g. \cite[Prop.~3.2]{EdGr:00}) implies
$$
G(G,G/B) = G(B, \mbox{pt}) = R(B) = R(T).
$$
Since we are working with rational coefficients, $R(T) \simeq \Q[u,u^{-1}]$
and this isomorphism can be chosen so that $u$ corresponds to
$[ {\mathcal O}_{\Pro^1}(2) ] $ in $G(G,\Pro^1)$.    We may view $R(G)$
as the subring $\Q[u+u^{-1}]$ of $R(T)$; then the ideal $I \subset
R(G)$ is generated by $u+u^{-1} - 2$, so 
$G(G,\Pro^1)/I = \Q[u,u^{-1}]/\langle (u-1)^2 \rangle$.  Also,
if $v = [ {\mathcal O}_{\Pro^1}(1) ]  \in G(\Pro^1)$, then $G(\Pro^1)
= \Q[v]/\langle (v-1)^2 \rangle$.  The forgetful map $G(G, \Pro^1) \to
G(\Pro^1)$ takes $u$ to $v^2$, and induces an isomorphism
$G(G,\Pro^1)/I \simeq G(\Pro^1)$.  However, if we were working
with integer coefficients, the forgetful map would not be surjective,
since in that case $v$ is not in the image. 
\end{example}

%\bibliographystyle{amsmath}
%\bibliography{refs}
%\bibliography{jabbrev,refs}
\def\cprime{$'$}

\end{document}